\documentclass[reqno,11pt]{amsart}
\usepackage{a4wide,color,eucal,enumerate,mathrsfs}
\usepackage[normalem]{ulem}
\usepackage{amsmath,amssymb,epsfig,amsthm} 
\usepackage[latin1]{inputenc}
\usepackage{psfrag}
\usepackage{mathtools}
\numberwithin{equation}{section}

\usepackage[english]{babel} 						
\usepackage[T1]{fontenc}							
\usepackage[utf8]{inputenx}

\usepackage{mathtools}								
\usepackage{empheq}									
\usepackage{enumitem}								

\usepackage{amssymb}								
\usepackage[sc]{mathpazo}							
\usepackage{mathrsfs}								

\usepackage{todonotes}								
\usepackage{aliascnt}								
\usepackage{comment}

\usepackage{epsfig}
\usepackage[normalem]{ulem}

\usepackage[babel]{csquotes}							

\usepackage{nameref}									
\usepackage[colorlinks	=	true,						
          	 	linkcolor	=	red,
            	urlcolor	=	red,
            	citecolor	=	red]%
            	{hyperref}
\usepackage{bookmark}									
\usepackage{cleveref}									

\newtheorem{theorem}{Theorem}[section]

\newtheorem{lemma}[theorem]{Lemma}
\newtheorem{proposition}[theorem]{Proposition}

\newtheorem{question}[theorem]{Question}
\newtheorem{definition}[theorem]{Definition}

\theoremstyle{remark}

\DeclareMathOperator{\diam}{diam}

\def\rr{{\mathbb R}}
\def\rn{{{\rr}^n}}

\def\fz{\infty}

\def\boz{{\Omega}}

\def\bint{{\ifinner\rlap{\bf\kern.35em--}
\int\else\rlap{\bf\kern.45em--}\int\fi}\ignorespaces}

\def\bbint{{\ifinner\rlap{\bf\kern.35em--}
\hspace{0.078cm}\int\else\rlap{\bf\kern.45em--}\int\fi}\ignorespaces}

\def\diam{{\mathop\mathrm{\,diam\,}}}

\def\r{\right}
\def\lf{\left}

\def\bint{{\ifinner\rlap{\bf\kern.35em--}
\int\else\rlap{\bf\kern.45em--}\int\fi}\ignorespaces}

\title[Zero volume boundary for extension domains from Sobolev to $BV$]{Zero volume boundary for extension domains\\ from Sobolev to $BV$}
\author{Tapio Rajala and Zheng Zhu }
\address{University of Jyv\"as\-kyl\"a \\
         Department of Mathematics and Statistics \\
         P.O. Box 35 (MaD) \\
         FI-40014 University of Jyv\"as\-kyl\"a \\
         Finland}
\email{tapio.m.rajala@jyu.fi; zheng.z.zhu@jyu.fi}

\date{April 2022}

\begin{document}

\begin{abstract}
In this note, we prove that the boundary of a $(W^{1, p}, BV)$-extension domain is of volume zero under the assumption that the domain $\boz$ is $1$-fat at almost every $x\in\partial\boz$. Especially, the boundary of any planar $(W^{1, p}, BV)$-extension domain is of volume zero. 
\end{abstract}

\maketitle

\section{Introduction}
Given $1\leq q\leq p\leq\fz$, a bounded domain $\boz\subset\rn$, $n\geq 2$, is said to be a $(W^{1, p}, W^{1, q})$-extension domain if there exists a bounded extension operator
\[
E\colon W^{1,p}(\boz)\mapsto W^{1,q}(\rn)
\]
and is said to be a $(W^{1, p}, BV)$-extension domain if there exists a bounded extension operator 
\[
E\colon W^{1, p}(\boz)\mapsto BV(\rn).
\]
The theory of Sobolev extensions is of interest in several fields in analysis. Partial motivations for the study of Sobolev extensions comes from the thoery of PDEs, for example, see \cite{Mazyabook}. It was proved in \cite{calderon,stein} that for every Lipschitz domain in $\rn$, there exists a bounded linear extension operator $E\colon W^{k, p}(\boz)\mapsto W^{k, p}(\rn)$ for each $k\in\mathbb N$ and $1\leq p\leq\fz$. Here $W^{k, p}(\boz)$ is the Banach space of all $L^p$-integrable functions whose distributional derivatives up to order $k$ are $L^p$-integrable. Later, the notion of $(\epsilon, \delta)$-domains was introduced by Jones in \cite{Jones}, and it was proved that for every $(\epsilon, \delta)$-domain, there exists a bounded linear extension operator $E\colon W^{k, p}(\boz)\mapsto W^{k, p}(\rn)$ for every $k\in\mathbb N$ and $1\leq p\leq\fz$.

In \cite{VGL}, a geometric characterization of planar $(W^{1, 2}, W^{1, 2})$-extension domain was given. By later results in \cite{PekkaJFA,Shvartsman,KRZ1,KRZ2}, we now have geometric characterizations of planar simply connected $(W^{1,p}, W^{1, p})$-extension domains for all $1\leq p\leq\fz$. A geometric characterization is also known for planar simply connected $(L^{k, p}, L^{k, p})$-extension domains with $2<p\leq\fz$, see \cite{ShvartsmanAdv,Whitney,Zobin}. Here $L^{k, p}(\boz)$ denotes the homogeneous Sobolev space which contains locally integrable functions whose $k$-th order distributional derivative is $L^p$-integrable. Beyond the planar simply connected case, geometric characterizations of Sobolev extension domains are still missing. However, several necessary properties have been obtained for general Sobolev extension domains.

For a measurable subset $F\subset\rn$, we use $|F|$ to denote its Lebesgue measure. In \cite{HKT2008:B, HKT2008}, Haj\l{}asz, Koskela and Tuominen proved for $1\leq p<\fz$ that a $\lf(W^{1, p}, W^{1,p}\r)$-extension domain $\boz\subset\rn$ must be Ahlfors regular which means that there exists a positive constant $C>1$ such that for every $x\in\overline\boz$ and $0<r<\min\lf\{1, \frac{1}{4}\diam\boz\r\}$, we have 
\begin{equation}\label{eq:regular}
|B(x, r)|\leq C|B(x, r)\cap\boz|.
\end{equation}
From the results in \cite{Kos:planarBV, Tapio}, we know that also $(BV, BV)$-extension domains are Ahlfors regular. For Ahlfors regular domains, the Lebesgue differentiation theorem then easily implies $|\partial\boz|=0$.

 In the case where $\boz$ is a planar Jordan $\lf(W^{1, p}, W^{1,p}\r)$-extension domain, $\boz$ has to be a so-called John domain when $1\leq p\leq 2$ and the complementary domain has to be John when $2\leq p<\fz$. The John condition implies that the Hausdorff dimension of $\partial\boz$ must be strictly less than 2, see \cite{KRmathann}. Recently, Lu$\check{c}$i\'c, Takanen and the first named author gave a sharp estimate on the Hausdorff dimension of $\partial\boz$, see \cite{LRTacv}. In general, the Hausdorff dimension of a $(W^{1, p}, W^{1, p})$-extension domain can well be $n$. 
 

The outward cusp domain with a polynomial type singularity is a typical example which is not a $(W^{1, p}, W^{1, p})$-extension domain for $1\leq p<\fz$. However, it is a $(W^{1, p}, W^{1, q})$-extension domain, for some $1\leq q<p\leq\fz$, see the monograph \cite{MPbook} and the references therein. Hence, for $1\leq q<p\leq\fz$, it is not necessary for a $(W^{1, p}, W^{1, q})$-extension domain to be Ahlfors regular. In the absence of Ahlfors regularity, one has to find alternative approaches for proving  $|\partial\boz|=0$. The first approach in \cite{ukhlov1,ukhlov2} was to generalize the Ahlfors regularity \eqref{eq:regular} to a Ahlfors-type estimate
\begin{equation}\label{eq:regular2}
    |B(x, r)|^p\leq C\Phi^{p-q}(B(x, r))|B(x, r)\cap\boz|^q
\end{equation}
for $(W^{1, p}, W^{1, p})$-extension domains with $n<q<p<\fz$. Here $\Phi$ is a bounded and quasiadditive set function generated by the $(W^{1, p}, W^{1, q})$-extension property and defined on open sets $U\subset\rn$, see Section \ref{sec:set}. By differentiating $\Phi$ with respect to the Lebesgue measure, one concludes that $|\partial\boz|=0$ if $\boz$ is
a $(W^{1, p}, W^{1, q})$-extension domain for $n < q < p < \fz$. Recently, Koskela, Ukhlov and the second named author \cite{KUZ} generalized this result 
and proved that the boundary of a $\lf(W^{1, p}, W^{1, q}\r)$-extension domain must be of volume zero for $n-1<q< p<\fz$ (and for $1\leq q< p<\fz$ on the plane).
For $1\leq q<n-1$ and $(n-1)q/(n-1-q)<p<\fz$, they constructed a $\lf(W^{1, p}, W^{1, q}\r)$-extension domain $\boz\subset\rn$ with $|\partial\boz|>0$. For the remaining range of exponents where $1\leq q\leq n-1$ and $q<p\leq(n-1)q/(n-1-q)$, it is still not clear whether the boundary of every $\lf(W^{1, p}, W^{1, q}\r)$-extension domain must be of volume zero. 

As is well-known, for every domain $\boz\subset\rn$, the space of functions of bounded variation $BV(\boz)$ strictly contains every Sobolev space $W^{1, q}(\boz)$ for $1\leq q\leq \fz$. Hence, the class of $\lf(W^{1, p}, BV\r)$-extension domains contains the class of $\lf(W^{1, p}, W^{1, q}\r)$-extension domains for every $1\leq q\leq p<\fz$. As a basic example to indicate that the containment is strict when $n \ge 2$, we can take the slit disk (the unit disk minus a radial segment) in the plane. The slit disk is a $\lf(W^{1, p}, BV\r)$-extension domain for every $1\leq p<\fz$, and even a $(BV, BV)$-extension domain; however it is not a $\lf(W^{1, p}, W^{1, q}\r)$-extension domain for any $1\leq q\leq p<\fz$. This basic example also shows that it is natural to consider the geometric properties of $\lf(W^{1, p}, BV\r)$-extension domains.
In this paper, we focus on the question whether the boundary of a $(W^{1, p}, BV)$-extension domain is of volume zero. 
Our first theorem tells us that the $(BV, BV)$-extension property is equivalent to the $\lf(W^{1, 1},  BV\r)$-extension property. Hence, a $(W^{1, 1}, BV)$-extension domain is Ahlfors regular and so its boundary is of volume zero.
\begin{theorem}\label{thm:1.1}
  A domain $\boz\subset\rn$ is a $(BV, BV)$-extension domain if and only if it is a $\lf(W^{1, 1}, BV\r)$-extension domain.
\end{theorem}
Since, $W^{1,1}(\boz)$ is a proper subspace of $BV(\boz)$ with $\|u\|_{W^{1,1}(\boz)}=\|u\|_{BV(\boz)}$ for every $u\in W^{1, 1}(\boz)$, $(BV, BV)$-extension property implies $(W^{1, 1}, BV)$-extension property immediately. The other direction from $(W^{1, 1}, BV)$-extension property to $(BV, BV)$-extension property is not as straightforward, as $W^{1, 1}(\boz)$ is only a proper subspace of $BV(\boz)$. The essential tool here is the Whitney smoothing operator constructed by Garc\'ia-Bravo and the first named author in \cite{Tapio}. This Whitney smoothing operator maps every function in $BV(\boz)$ to a function in $W^{1, 1}(\boz)$ with the same trace on $\partial\boz$, so that the norm of the image in $W^{1, 1}(\boz)$ is uniformly controlled from above by the norm of the corresponding preimage in $BV(\boz)$.

With an extra assumption that $\boz$ is $q$-fat at almost every point on the boundary $\partial\boz$, in \cite{KUZ} it was shown that the boundary of a $(W^{1,p}, W^{1, q})$-extension domain is of volume zero when $1\leq q<p<\fz$. The essential point there was that the $q$-fatness of the domain on the boundary guarantees the continuity of a $W^{1, q}$-function on the boundary. Maybe a bit surprisingly, the assumption that the domain is $1$-fat at almost every point on the boundary also guarantees that the boundary of a $(W^{1, p}, BV)$-extension domain is of volume zero. In particular, every planar domain is $1$-fat at every point of the boundary. Hence, we have the following theorem.
\begin{theorem}\label{thm:Rn}
  Let $\boz\subset\rn$ be a $(W^{1, p}, BV)$-extension domain for $1\leq p<\fz$, which is $1$-fat at almost every $x\in\partial\boz$. Then $|\partial\boz|=0$. In particular, for every planar $(W^{1, p}, BV)$-extension domain $\boz$ with $1\leq p<\fz$, we have $|\partial\boz|=0$.
\end{theorem}


In light of the results and example given in \cite{KUZ}, the most interesting open question is what happens in the 
range $1<p\leq (n-1)/(n-2)$ of exponents. For this range, we do not know whether the boundary of a $(W^{1, p}, BV)$-extension domain must be of volume zero. If a counterexample exists in this range, it might be easier to construct it in the $(W^{1, p}, BV)$-case rather than the $(W^{1, p},W^{1,1})$-case. Hence we leave it as a question here.
\begin{question}
For $1<p\leq(n-1)/(n-2)$, is the boundary of a $(W^{1, p}, BV)$-extension domain of volume zero?
\end{question}

\section{Preliminaries}
 For a locally integrable function $u\in L^1_{\rm loc}(\boz)$ and a measurable subset $A\subset\boz$ with $0<|A|<\fz$, we define
\[u_A:=\bint_Eu(y)\,dy=\frac{1}{|A|}\int_Au(y)\,dy.\]
\begin{definition}\label{de:Sobolev}
   Let $\boz\subset\rn$ be a domain. For every $1\leq p\leq\fz$, we define the Sobolev space $W^{1, p}(\boz)$ to be 
   \[W^{1, p}(\boz):=\lf\{u\in L^p(\boz): \nabla u\in L^p(\boz;\rn)\r\},\]
   where $\nabla u$ denotes the distributional gradient of $u$. It is equipped with the nonhomogeneous norm
   \[\|u\|_{W^{1, p}(\boz)}=\|u\|_{L^p(\boz)}+\|\nabla u\|_{L^p(\boz)}.\]
\end{definition}
Now, let us give the definition of functions of bounded variation.
\begin{definition}\label{de:BV}
   Let $\boz\subset\rn$ be a domain. A function $u\in L^1(\boz)$ is said to have bounded variation and denoted $u\in BV(\boz)$ if 
   \[\|Du\|(\boz):=\sup\lf\{\int_\boz f{\rm div}(\phi) dx:\phi\in C^1_o(\boz;\rn), |\phi|\leq 1\r\}<\fz.\]
   The space $BV(\boz)$ is equipped with the norm
   \[\|u\|_{BV(\boz)}:=\|u\|_{L^1(\boz)}+\|Du\|(\boz).\]
\end{definition}

\begin{definition}\label{de:extension}
   We say that a domain $\Omega\subset\rn$ is a $\lf(W^{1, p}, BV\r)$-extension domain for $1\leq p<\fz$, if there exists a bounded extension operator $E\colon W^{1, p}(\boz)\mapsto BV(\rn)$ such that for every $u\in W^{1, p}(\boz)$, we have $E(u)\in BV(\rn)$ with $E(u)\big|_\boz\equiv u$ and 
   \[\|E(u)\|_{BV(\rn)}\leq C\|u\|_{W^{1, p}(\boz)}\]
   for a constant $C>1$ independent of $u$.
\end{definition}

Let $U\subset\rn$ be an open set and $A\subset U$ be a measurable subset with $\overline A\subset U$. The $p$-admissible set $\mathcal W_p(A; U)$ is defined by setting 
\[\mathcal W_p(A; U):=\lf\{u\in W^{1, p}_0(U)\cap C(U):u\big|_{A}\geq 1\r\}.\]
\begin{definition}\label{de:capacity}
   Let $U\subset\rn$ be an open set and $A\subset U$ with $\overline{A}\subset U$. The relative $p$-capacity $Cap_p(A; U)$ is defined by setting 
   \[Cap_p(A; U):=\inf_{u\in\mathcal W_p(A;U)}\int_{U}|\nabla u(x)|^p\,dx.\]
\end{definition}
Following Lahti \cite{Panu}, we define $1$-fatness below.
\begin{definition}\label{de:1fat}
   Let $A\subset\rn$ be a measurable subset. We say that $A$ is $1$-thin at the point $x\in\rn$, if 
   \[\lim_{r\to0^+}r\frac{Cap_1\lf(A\cap B(x, r); B(x, 2r)\r)}{\lf|B(x, r)\r|}=0.\]
   If $A$ is not $1$-thin at $x$, we say that $A$ is $1$-fat at $x$.
   Furthermore, we say that a set $U$ is $1$-finely open, if $\rn\setminus U$ is $1$-thin at every $x\in U$. 
\end{definition}
By \cite[Lemma 4.2]{Panu}, the collection of $1$-finely open sets is a topology on $\rn$. For a function $u\in BV(\rn)$, we define the lower approximate limit $u^\star$ by setting 
\[u_\star(x):=\sup\lf\{t\in\overline{\rr}:\lim_{r\to 0^+}\frac{\lf|B(x, r)\cap\{u<t\}\r|}{\lf|B(x, r)\r|}=0\r\}\]
and the upper approximate limit $u_\star$ by setting 
\[u^\star(x):=\inf\lf\{t\in\overline\rr:\lim_{r\to 0^+}\frac{\lf|B(x, r)\cap\{u>t\}\r|}{\lf|B(x, r)\r|}=0\r\}.\]
The set 
\[S_u:=\lf\{x\in\rn: u_\star(x)<u^\star(x)\r\}\]
is called the jump set of $u$. By the Lebesgue differentiation theorem, $\lf|S_u\r|=0$. Using the lower and upper approximate limits, we define the precise representative $\tilde u:=(u^\star+u_\star)/2$. The following lemma was proved in \cite[Corollary 5.1]{Panu}. 
\begin{lemma}\label{lem:1-fine}
   Let $u\in BV(\rn)$. Then $\tilde u$ is $1$-finely continuous at $\mathcal H^{n-1}$-almost every $x\in\rn\setminus S_u$.
\end{lemma}
The following lemma for $u\in W^{1, 1}(\rn)$ was proved in \cite[Lemma 2.6]{KUZ}, which is also a corollary of a result in \cite{Tero}. We generalize it to $BV(\rn)$ here.
\begin{lemma}\label{le:fat}
   Let $\boz\subset\rn$ be a domain which is $1$-fat at almost every point $x\in\partial\boz$. If $u\in BV(\rn)$ with $u\big|_{B(x, r)\cap\boz}\equiv c$ for some $x\in\partial\boz$, $0<r<1$ and $c\in\rr$. Then $u(y)=c$ for almost every $y\in B(x, r)\cap\partial\boz$.
\end{lemma}
\begin{proof}
   Let $u\in BV(\rn)$ satisfy the assumptions. Then the precise representative $\tilde u\big|_{B(x, r)\cap\boz}\equiv c$. Since $|S_u|=0$, by Lemma \ref{lem:1-fine}, there exists a subset $N_1\subset\rn$ with $|N_1|=0$ such that $\tilde u$ is $1$-finely continuous on $\rn\setminus N_1$. By the assumption, there exists a measure zero set $N_2\subset\partial\boz$ such that $\boz$ is $1$-fat on $\partial\boz\setminus N_2$. By \Cref{de:1fat}, one can see that $B(x, r)\cap\boz$ is also $1$-fat on $(B(x, r)\cap\partial\boz)\setminus N_2$. For every $y\in(B(x, r)\cap\partial\boz)\setminus(N_1\cup N_2)$, since $\tilde u$ is $1$-finely continuous on it and any $1$-fine neighborhood of $y$ must intersect $B(x, r)\cap\boz$, we have $\tilde u(y)=c$. Hence $u(y)=c$ for almost every $y\in B(x, r)\cap\partial\boz$.
\end{proof}

The following coarea formula for $BV$ functions can be found in \cite[Section 5.5]{Evans}. See also \cite[Theorem 2.2]{Tapio}.
\begin{proposition}
    Given a function $u\in BV(\boz)$, the superlevel sets $u_t=\{x\in\boz:u(x)>t\}$ have finite perimeter in $\boz$ for almost every $t\in\rr$ and 
    \[\|Du\|(F)=\int_{-\fz}^\fz P(u_t, F)\,dt\]
    for every Borel set $F\subset\boz$. Conversely, if $u\in L^1(\boz)$ and 
    \[\int_{-\fz}^\fz P(u_t,\boz)\,dt<\fz\]
    then $u\in BV(\boz)$.
\end{proposition}

See \cite[Theorem 3.44]{Fusco} for the proof of the following $(1, 1)$-Poincar\'e inequality for $BV$ functions.
\begin{proposition}\label{prop:BVpoin}
    Let $\boz\subset\rn$ be a bounded Lipschitz domain. Then there exists a constant $C>0$ depending on $n$ and $\boz$ such that for every $u\in BV(\boz)$, we have 
    \[\int_\boz|u(y)-u_\boz|\,dy\leq C\|Du\|(\boz).\]
    In particular, there exists a constant $C>0$ only depending on $n$ so that if $Q, Q'\subset\rn$ are two closed dyadic cubes with $\frac{1}{4}l(Q')\leq l(Q)\leq 4l(Q')$ and $\boz:={\rm int}(Q\cup Q')$ connected, then for every $u\in BV(\boz)$,
    \begin{equation}\label{eq:poincare}
        \int_\boz|u(y)-u_\boz|\,dy\leq Cl(Q)\|Du\|(\boz).
    \end{equation}
\end{proposition}

\section{A set function arising from the extension}\label{sec:set}
In this subsection, we introduce a set function defined on the class of open sets in $\rn$ and taking nonnegative values. Our set function here is a modification of the one originally introduced by Ukhlov \cite{ukhlov1, ukhlov2}. See also \cite{VoUk1, VoUk2} for related set functions.
The modified version of the set function we use is from \cite{KUZ}, where it was used by Koskela, Ukhlov and the second named author to study the size of the boundary of a $(W^{1, p}, W^{1, q})$-extension domains. Let us recall that a set function $\Phi$ defined on the class of open subsets of $\rn$ and taking nonnegative values is called quasiadditive (see for example \cite{VoUk1}), if for all open sets $U_1\subset U_2\subset\rn$, we have 
\[\Phi(U_1)\leq\Phi(U_2),\]
and there exists a positive constant $C$ such that for arbitrary pairwise disjoint open sets $\{U_i\}_{i=1}^\fz$, we have 
\begin{equation}\label{eq:quasiadditive}
    \sum_{i=1}^\fz\Phi(U_i)\leq C\Phi\lf(\bigcup_{i=1}^\fz U_i\r).
\end{equation}

Let $\boz\subset\rn$ be a $(W^{1, p}, BV)$-extension domain for some $1<p<\fz$. For every open set $U\subset\rn$ with $U\cap\boz\neq\emptyset$, we define 
\[W^p_0(U, \boz):=\lf\{u\in W^{1, p}(\boz)\cap C(\boz): u\equiv 0\ {\rm on}\ \boz\setminus U\r\}.\]
For every $u\in W^p_0(U, \boz)$, we define 
\[\Gamma(u):=\inf\lf\{\|Dv\|(U): v\in BV(\rn), v\big|_{\boz}\equiv u\r\}.\]
Then we define the set function $\Phi$ by setting 
\begin{equation}\label{eq:setfunction}
   \Phi(U):=\begin{cases}\sup_{u\in W^p_0(U, \boz)}\lf(\frac{\Gamma(u)}{\|u\|_{W^{1, p}(U\cap\boz)}}\r)^{k}, \ \ \text{with } \frac{1}{k}=1-\frac{1}{p},
   & \text{if }U\cap\boz\neq\emptyset,\\
   0, & \text{otherwise.}
   \end{cases}
\end{equation}
The proof of the following lemma is almost the same as the proof of \cite[Theorem 3.1]{KUZ}. One needs to simply replace $\|Dv\|_{L^q(U)}$ by $\|Dv\|(U)$ in the proof of \cite[Theorem 3.1]{KUZ} and repeat the argument. 
\begin{lemma}\label{le:setfunction}
   Let $1<p<\fz$ and let $\boz\subset\rn$ be a bounded $(W^{1, p}, BV)$-extension domain. Then the set function defined in (\ref{eq:setfunction}) for all open subsets of $\rn$ is bounded and quasiadditive.
\end{lemma}
The upper and lower derivatives of a quasiadditive set function $\Phi$ are defined by setting 
\[\overline{D\Phi}(x):=\limsup_{r\to0^+}\frac{\Phi(B(x, r))}{|B(x, r)|}\quad {\rm and}\quad \underline{D\Phi}(x) = \liminf_{r\to 0^+}\frac{\Phi(B(x,r))}{|B(x, r)|}.\]
By \cite{Rado,VoUk1}, we have the following lemma. See also \cite[Lemma 3.1]{KUZ}.
\begin{lemma}\label{le:setFbound}
  Let $\Phi$ be a bounded and quasiadditive set function defined on open sets $U\subset\rn$. Then $\overline{D\Phi}(x)<\fz$ for almost every $x\in\rn$.
\end{lemma}

The following lemma immediately comes from the definition (\ref{eq:setfunction}) for the set function $\Phi$.
\begin{lemma}\label{cor:extension}
 Let $1<p<\fz$ and let $\boz\subset\rn$ be a bounded $(W^{1, p}, BV)$-extension domain. Then, for a ball $B(x, r)$ with $x\in\partial\boz$ and every function $u\in W^p_0(B(x, r), \boz)$, there exists a function $v\in BV(B(x, r))$ with $v\big|_{B(x, r)\cap\boz}\equiv u$ and 
 \begin{equation}\label{eq:extension}
     \|Dv\|(B(x, r))\leq 2\Phi^{\frac{1}{k}}(B(x, r))\|u\|_{W^{1, p}(B(x, r)\cap\boz)}, \ \ {\rm where}\ \ \frac{1}{k}=1-\frac{1}{p}.
 \end{equation}
\end{lemma}

\section{Proofs of the results}
In this section we prove Theorems \ref{thm:1.1} and \ref{thm:Rn}.
\begin{proof}[Proof of \Cref{thm:1.1}]
   Let us first assume that $\boz\subset\rn$ is a $(BV,BV)$-extension domain with the extension operator $E$. Since $W^{1,1}(\boz)\subset BV(\boz)$ with $\|u\|_{BV(\boz)}=\|\nabla u\|_{L^1(\boz)}$ for every $u\in W^{1, 1}(\boz)$, we have $$\|E(u)\|_{BV(\boz)}\leq C\|u\|_{BV(\boz)}\leq C\|u\|_{W^{1, 1}(\boz)}.$$
   This implies that $\boz$ is a $\lf(W^{1,1}, BV\r)$-extension domain with the same operator $E$ restricted to $W^{1,1}(\Omega)$. 
   
   Let us then prove the converse and assume that
    $\boz\subset\rn$ is a $(W^{1,1}, BV)$-extension domain with an extension operator $E$. Let $S_{\boz,\boz}$ be the Whitney smoothing operator defined in \cite{Tapio}. Then by \cite[Theorem 3.1]{Tapio}, for every $u\in BV(\boz)$, we have $S_{\boz,\boz}(u)\in W^{1, 1}(\boz)$ with 
   \[\|S_{\boz,\boz}(u)\|_{W^{1, 1}(\boz)}\leq C\|u\|_{BV(\boz)}\]
   for a positive constant $C$ independent of $u$, and 
   \begin{equation}\label{eq:boundary}
   \|D(u-S_{\boz, \boz}(u))\|(\partial\boz)=0,
   \end{equation}
   where $u-S_{\boz,\boz}(u)$ is understood to be defined on the whole space $\rn$ via a zero-extension. Then $E(S_{\boz, \boz}(u))\in BV(\rn)$ with 
   \[\|E(S_{\boz, \boz}(u))\|_{BV(\rn)}\leq C\|S_{\boz, \boz}(u)\|_{W^{1, 1}(\boz)}\leq C\|u\|_{BV(\boz)}.\]
   Now, define $T\colon BV(\Omega) \to BV(\rn)$ by setting for every $u \in BV(\Omega)$
   \[T(u)(x):=\begin{cases}
   u(x),\ \ {\rm if}\ \ x\in \boz\\
   E(S_{\boz, \boz}(u))(x),\ \ {\rm if}\ \ x\in\rn\setminus\boz.
   \end{cases}\]
   By (\ref{eq:boundary}), we have $T(u)\in BV(\rn)$ with 
   \[\|T(u)\|_{BV(\rn)}\leq\|E(S_{\boz,\boz}(u))\|_{BV(\rn)}+\|u\|_{BV(\boz)}\leq C\|u\|_{BV(\boz)}.\]
   Hence, $\boz$ is a $BV$-extension domain.
\end{proof}

\begin{proof}[Proof of \Cref{thm:Rn}]
Assume towards a contradiction that $|\partial\boz|>0$.  By the Lebesgue density point theorem and \Cref{le:setFbound}, there exists a measurable subset $U$ of $\partial\boz$ with $|U|=|\partial\boz|$ such that every $x\in U$ is a Lebesgue point of $\partial\boz$ and $\overline{D\Phi}(x)<\fz$. Fix $x\in U$. Since $x$ is a Lebesgue point, there exists a sufficiently small $r_x>0$, such that for every $0<r<r_x$, we have 
\[\lf|B(x, r)\cap\overline\boz\r|\geq\frac{1}{2^{n-1}}\lf|B(x, r)\r|.\]
Let $r\in(0, r_x)$ be fixed. Since $|\partial B(x, s)|=0$ for every $s\in (0, r)$, we have
\begin{equation}\label{eq:volume1}
    \lf|B\lf(x, \frac{r}{4}\r)\cap\overline\boz\r|\geq\frac{1}{2^{n-1}}\lf|B\lf(x, \frac{r}{4}\r)\r|\geq\frac{1}{2^{3n-1}}\lf|B(x, r)\r|
\end{equation}
and 
\begin{equation}\label{eq:volume2}
    \lf|\lf(B(x,r)\setminus B\lf(x, \frac{r}{2}\r)\r)\cap\overline\boz\r|\geq |B(x, r)\cap\overline\boz|-\lf|B\lf(x, \frac{r}{2}\r)\r|\geq\frac{1}{2^n}|B(x, r)|.
\end{equation}
Define a test function $u\in W^{1, p}(\boz)\cap C(\boz)$ by setting 
  \begin{equation}\label{eq:testF}
     u(y):=\begin{cases}
     1,\ \ &{\rm if}\ y\in B\lf(x, \frac{r}{4}\r)\cap\boz,\\
     \frac{-4}{r}|y-x|+2,\ \ &{\rm if}\ y\in \lf(B\lf(x, \frac{r}{2}\r)\setminus B\lf(x, \frac{r}{4}\r)\r)\cap\boz,\\
     0,\ \ &{\rm if}\ y\in\boz\setminus B\lf(x, \frac{r}{2}\r).
     \end{cases}
  \end{equation}
 We have 
\begin{equation}\label{eq:upperbound}
    \lf(\int_{B(x, r)\cap\boz}|u(y)|^p+|\nabla u(y)|^p\,dx\r)^{\frac{1}{p}}\leq\frac{C}{r}|B(x, r)\cap\boz|^{\frac{1}{p}}.
\end{equation}
Since $u\equiv 0$ on $\boz\setminus B(x, r/2)$, we have $u\in W^p_0(B(x, r), \boz)$. Then, by the definition \eqref{eq:setfunction} of the set function $\Phi$ and by \Cref{cor:extension}, there exists a function $v\in BV(B(x, r))$ with $v\big|_{B(x, r)\cap\boz}\equiv u$ and  
\begin{equation}\label{eq:extenbound}
\|Dv\|(B(x, r))\leq 2\Phi^{\frac{1}{k}}(B(x, r))\|u\|_{W^{1, p}(B(x, r)\cap\boz)}.
\end{equation}
By the Poincar\'e inequality of $BV$ functions stated in \Cref{prop:BVpoin}, we have 
\begin{equation}\label{eq:BVpoincare}
    \int_{B(x, r)}|v(y)-v_{B(x, r)}|\,dy\leq C r\|Dv\|(B(x, r)).
\end{equation}
Since $\boz$ is $1$-fat on almost every $z\in\partial\boz$, by \Cref{le:fat}, $v(z)=1$ for almost every $z\in B\lf(x,\frac{r}{4}\r)\cap\partial\boz$ and $v(z)=0$ for almost every $z\in \lf(B(x,r)\setminus B\lf(x, \frac{r}{2}\r)\r)\cap\partial\boz$. Hence, on one hand, if $v_{B(x, r)}\leq\frac{1}{2}$, we have 
\[\int_{B(x, r)}|v(y)-v_{B(x, r)}|\,dy\geq\frac{1}{2}\lf|B\lf(x, \frac{r}{4}\r)\cap\overline\boz\r|\geq c|B(x, r)|.\]
On the other hand, if $v_{B(x, r)}>\frac{1}{2}$, we have 
\[\int_{B(x, r)}|v(y)-v_{B(x, r)}|\,dy\geq\frac{1}{2}\lf|\lf(B(x, r)\setminus B\lf(x, \frac{r}{2}\r)\r)\cap\overline\boz\r|>c|B(x, r)|.\]
All in all, we always have 
\begin{equation}\label{eq:lowerbound}
    \int_{B(x, r)}|v(y)-v_{B(x, r)}|\,dy\geq c|B(x, r)|
\end{equation}
for a sufficiently small constant $c>0$. Thus, by combining inequalities (\ref{eq:upperbound})-(\ref{eq:lowerbound}), we obtain
\[\Phi(B(x, r))^{p-1}|B(x, r)\cap\boz|\geq c|B(x, r)|^p\]
for a sufficiently small constant $c>0$. This gives 
  \[|B(x, r)\cap\partial\boz|\leq |B(x, r)|-|B(x, r)\cap\boz|\leq |B(x, r)|-C\frac{|B(x, r)|^p}{\Phi(B(x, r))^{p-1}}.\]
 Since $\overline{D\Phi}(x)<\fz$, we have 
 \begin{eqnarray}
     \limsup_{r\to0^+}\frac{|B(x, r)\cap\partial\boz|}{|B(x, r)|}&\leq&\limsup_{r\to0^+}\lf(1-\frac{|B(x, r)\cap\boz|}{|B(x, r)|}\r)\nonumber\\
         &\leq&\limsup_{r\to0^+}\lf(1-\frac{|B(x, r)|^{p-1}}{\Phi(B(x, r))^{p-1}}\r)\leq1-c\overline{D\Phi}(x)^{1-p}<1.\nonumber
  \end{eqnarray}
 This contradicts the assumption that $x$ is a Lebesgue point of $\partial\boz$. Hence, we conclude that $|\partial\boz|=0$.
 
 Let us then consider the case $\boz\subset\rr^2$. By \cite[Theorem A.29]{HenclandPekka}, for every $x\in\partial\boz$ and every $0<r<\min\lf\{1, \frac{1}{4}\diam(\boz)\r\}$, we have
 \[Cap_1(\boz\cap B(x, r); B(x, 2r))\geq cr\]
 for a constant $0<c<1$. This implies that $\boz$ is $1$-fat at every $x\in\partial\boz$. Hence, by combining this with the first part of the theorem, we have that the boundary of any planar $(W^{1, p}, BV)$-extension domain is of volume zero.
\end{proof}
\bibliographystyle{alpha}
\bibliography{Bibliography} 
\end{document}